\documentclass[10pt,a4paper]{amsart}
\usepackage{amsfonts}
\usepackage{amssymb}
\usepackage{amsthm}
\usepackage[english]{babel}
\usepackage{hhline}
\usepackage[ansinew]{inputenc}
\usepackage{amsmath}
\usepackage[all,2cell,ps]{xy}
\usepackage{qsymbols}
\usepackage{color}
\usepackage{epsfig}
\usepackage{color}
\usepackage{graphics}
\usepackage{graphicx}%
\usepackage{enumerate}

\usepackage{amssymb}

\theoremstyle{plain}
\newtheorem{thm}{Theorem}
\newtheorem{cor}[thm]{Corollary}
\newtheorem{lem}[thm]{Lemma}
\newtheorem{prop}[thm]{Proposition}
\theoremstyle{definition}
\newtheorem{defn}[thm]{Definition}
\newtheorem{rem}{Remark}

\newcommand{\fro}{(\ )^{\FIS}}
\newcommand{\X}{X}

\newcommand{\U}{\mathrm{U}}
\newcommand{\f}{(\ )^{\IS}}

\newcommand{\Id}{\mathsf{Id}}
\newcommand{\ra}{\rightarrow}
\newcommand{\we}{\wedge}
\newcommand{\Ra}{\Rightarrow}

\newcommand{\Fil}{\mathsf{Fil}}

\newcommand{\FC}{\mathrm{L}}

\newcommand{\fIS}{f^{\mathsf{IS}}}
\newcommand{\hIS}{h^{\mathsf{IS}}}

\newcommand{\Hil}{\mathsf{Hil}}
\newcommand{\FHil}{\mathsf{FHil}}
\newcommand{\IS}{\mathsf{IS}}
\newcommand{\FIS}{\mathsf{FIS}}
\newcommand{\SIS}{\mathsf{IS_{S}}}
\newcommand{\gaIS}{\mathsf{IS_{\gamma}}}
\newcommand{\GabIS}{\mathsf{IS_{G}}}

\newcommand{\SHil}{\mathsf{Hil_{S}}}
\newcommand{\gaHil}{\mathsf{Hil_{\gamma}}}
\newcommand{\GabHil}{\mathsf{Hil_{G}}}

\newcommand{\tp}{\mathrm{\tau^{\pi}}}
\newcommand{\tup}{\mathrm{\tau_{1}^{\pi}}}
\newcommand{\tdp}{\mathrm{\tau_{2}^{\pi}}}
\newcommand{\Sp}{\mathrm{S^{\pi}}}
\newcommand{\Gp}{\mathrm{G^{\pi}}}
\newcommand{\gap}{\mathrm{\gamma^{\pi}}}

\usepackage{etoolbox}
\makeatletter
\patchcmd{\maketitle}
 {\def\@makefnmark}
 {\def\@makefnmark{}\def\useless@macro}
 {}{}
\makeatother

\begin{document}

\title{On the free frontal implicative semilattice extension of a frontal Hilbert algebra}

\author{Ramon Jansana and Hern\'an J.
San Mart\'{\i}n}

\maketitle

\begin{abstract}
In this paper we define a functor from the algebraic category of
frontal Hilbert algebras to the algebraic category of frontal
implicative semilattices which is  left adjoint to the forgetful
functor from the category of frontal implicative semilattices to
that of frontal Hilbert algebras.
\end{abstract}

\section{Introduction}

Frontal Heyting algebras were introduced by Esakia in
\cite{Esakia} as the algebraic models of the modalized  Heyting
calculus $mHC$. They are Heyting algebras together with a unary
modal box-like operation  with the algebraic properties of the
dual of the derivative operator of a topological  space when is
applied to the Heyting algebra of its open sets.

In \cite{CSM0} the notion of frontal operator is generalized to
Hilbert algebras, and hence to implicative semilattices too, as
well as the frontal operators $\gamma$, $S$ and $G$ considered by
Caicedo and Cignoli in \cite{Caici} (examples 3.1, 5.2 and 5.3
respectively) in the framework of Heyting algebras as particular
examples of implicitly definable compatible functions.

The variety of Hilbert algebras is the class of the subalgebras of
the reducts to the languge $\{\to, 1\}$ of the implicative
semilattices. In \cite{CJ2} an explicit definition of a left
adjoint of the forgetful functor from the category of implicative
semilattices to the category of Hilbert algebras is given. In
\cite{CSM} another explicit description of such an adjoint functor
was presented following an alternative path.

The main goal of this paper is to obtain an explicit definition of
a left adjoint functor to the forgetful functor from the category
of frontal implicative semilattices to the category of frontal
Hilbert algebras, thus providing for every frontal Hilbert algebra
a specific construction of its free extension to a frontal
implicative semilattice. Building on it we also obtain left
adjoint functors to the forgethful functor from the categories of
bounded implicative semilattices with a $\gamma$-function,
implicative semilattices with a successor function, and bounded
implicative semilattices with a Gabbay function to the categories
of bounded Hilbert algebras augmented with a $\gamma$-function,
Hilbert algebras augmented with a successor function, and bounded
Hilbert algebras augmented with a Gabbay function respectively.

The paper is organized as follows. In Section \ref{Sec:Prelim} we
introduce the preliminaries we need; they include the main
properties concerning frontal operators in Heyting algebras and
Hilbert algebras needed. We also present the explicit description
of a left adjoint to the forgetful functor from the category of
implicative semilattices to the category of Hilbert algebras that
we use in Section \ref{FHil} to obtain our left adjoint functor to
the forgetful functor from the category of frontal implicative
semilattices to the category of frontal Hilbert algebras. In
sections \ref{gaHil}, \ref{SHil} and \ref{GabHil} we use the
adjunctions presented in Section \ref{FHil} in order to obtain a
similar result for some categories of frontal implicative
semilattices and frontal Hilbert algebras determined by the
notions of $\gamma$-function, successor function, and a Gabbay
function.

\section{Preliminaries}\label{Sec:Prelim}

We start with the basic notions on posets we need along the paper.
Let $P = (P,\leq)$ be a poset. A subset $U\subseteq P$ is said to
be an \emph{upset} if for all $x, y \in P$ such that $x\in U$ and
$x\leq y$ we have $y\in U$. The notion of \emph{downset} is
defined dually. The \textit{upset generated} by a set $Y \subseteq
P$ is the set $[Y) = \{x\in P: (\exists y\in Y)\ y\leq x\}$ and
the downset generated by $Y$ the set $(Y] = \{x\in P:(\exists y\in
Y)\ x\leq y\}$. If $Y = \{y\}$, then we will write $[y)$ and $(y]$
instead of $[\{y\})$ and $(\{y\}]$, respectively. A set $I
\subseteq P$ is an \emph{order-ideal} if $I$ is a nonempty downset
that is up-directed, namely that for every $a, b \in I$ there is
$c \in I$ such that $a \leq c$ and $b \leq c$. Dually we have the
notion of order filter, but we will not use this notion in the
paper. We denote by $\Id(P)$ the set of order-ideals of $P$.

Given a set $X$ and a set $Y \subseteq X$ we denote the relative
complement of $Y$ to $X$, i.e., $X \setminus Y$, by $Y^{c}$, that
is $Y^{c}:=\{x \in X: x \not \in Y\}$. The context will always
make it clear with respect to which set we are taking the relative
complement.

We assume  the reader is familiar with the theory of Heyting
algebras \cite{BD,R} and propositional intuitionistic calculus, of
which the variety of Heyting algebras is the algebraic
counterpart. We recall that the lattice of all open sets of a
topological space $X$ is a Heyting algebra where the implication
from an open set $U$ to an open set $V$ is defined as the interior
of $U^{c} \cup V$. We denote that Heyting algebra by
$\mathcal{O}(X)$.

Let $(P,\leq)$ be a poset. We denote by $P^{+}$ the set of all
upsets of $(P,\leq)$ as well as the poset we obtain by ordering
this set by the inclusion relation. The set of all upsets of
$(P,\leq)$ is closed under intersections and unions of arbiratry
no-empty families and contains $P$ and $\emptyset$. Therefore we
have a complete distributive lattice. We can define the binary
operation $\Ra$ on $P^{+}$ by setting for every $U, V \in P^{+}$,
\begin{equation}\label{imp}
U\Ra V =(U\cap V^{c}]^{c}.
\end{equation}
Then, together  with $\Ra$,  the complete distributive lattice
$P^{+}$ is a Heyting algebra. In fact, we can look at $P^{+}$ as a
topology on $P$. In this topology $U \Ra V$ is precisely the
interior of $U^{c} \cup V$.

Hilbert algebras were introduced in the early 50's by Henkin for
some investigations of the implication of intuitionistic and other
non-classical logics (\cite{R}, pp.\ 16). In the 1960s, they were
studied especially by Horn \cite{H} and Diego \cite{D}.

\begin{defn}
A \emph{Hilbert algebra} is an algebra $H = (H,\ra,1)$ of type $(2,0)$
which satisfies the following conditions for every $a,b,c\in H$:
\begin{enumerate}[\normalfont 1)]
\item $a\ra (b\ra a) = 1$, \item $(a\ra (b\ra c)) \ra ((a\ra b)\ra
(a\ra c)) = 1$, \item if $a\ra b = b\ra a = 1$ then $a = b$.
\end{enumerate}
\end{defn}

In every Hilbert algebra $H$ we have a partial order $\leq $
defined by setting for every $a, b \in H$
\[
a\leq b \quad  \Longleftrightarrow \quad a\ra b = 1.
\]
This order is called \emph{the natural
order} of $H$. Relative to it, $1$ is the greatest (or top)
element.

In \cite{D} it was proved that the class of Hilbert algebras is a
variety. We write $\Hil$ for the variety of Hilbert algebras as well as
 for the the category whose objects are Hilbert
algebras and whose morphisms are the homomorphisms between them.
Several properties of Hilbert algebras can be found in \cite{B,D}.

A \emph{semilattice} is an algebra $(A, \we)$ of type (2) where
$\we$ is associative, commutative and idempotent. Given a
semilattice $(A, \we)$ the binary relation $\leq$ defined by
\[
a \leq  b \quad  \Longleftrightarrow \quad a \we b = a,
\]
for  every $a,b \in A$, is a partial order where any two elements
$a,b \in A$ have a greatest lower bound $a\we b$ (i.e., an
infimum). We say that a semilattice $(A, \we)$ is a
\textit{meet-semilattice} when we consider the partial order just
defined. Every poset $(P,\leq)$ with the property that any two
elements $a,b \in P$ have a greatest lower bound defines a
meet-semilattice by taking on $P$ the operation $\we$ defined by
setting for every $a, b \in P$ that $a \we b$ is the greatest
lower bound of $a, b$. The partial order of the meet-semilattice
$(P,\we)$ is the partial order $\leq$ we start with. A
meet-semilattice $(A, \we)$ is (upper) \textit{bounded} if $\leq$
has a greatest element, that we denote by $1$. Then we extend the
signature and we consider the algebra $(A, \we, 1)$. Throughout
this paper we just write \emph{semilattice} in place of
meet-semilattice.

Implicative semilattices were introduced in \cite{N}. For studies
of implicative semilattices we refer to \cite{KP,N1,N2}.
Implicative semilattices are the algebraic counterpart of the
$\{\we, \to, 1\}$-fragment of intuitionistic logic. They are a
combination of a Hilbert algebra and a semilattice where $a \to b$
is the meet-residual of $b$ by $a$.

\begin{defn}
An \emph{implicative semilattice} is an algebra $(H, \we, \to)$ of
type $(2,2)$ such that $(H,\we)$ is a semilattice and for every
$a,b,c\in H$,
\[
a\we b \leq c \quad \Longleftrightarrow \quad a\leq b\ra c,
\]
where $\leq$ is the semilattice order of $(H, \we)$.
\end{defn}

Every implicative semilattice has a greatest element, denoted by
$1$. In this paper we take the signature for implicative
semilattices to be $\{\we, \to, 1\}$ so that $(H, \we, \ra, 1)$ is
an implicative semilattice if $(H, \we, \ra)$ is one and $1$ is
its greatest element in the natural order. The class of
implicative semilattices is a variety, as it was proved by
Monteiro in \cite{Mon}. For more details about implicative
semilattices see \cite{Cu,N}.

It is known that implicative semilattices are the subalgebras of
the reducts of Heyting algebras to the language $\{\we, \to, 1\}$
and Hilbert algebras are the subalgebras of the $\{\to,
1\}$-reducts of implicative semilattices. Therefore, an arbitrary
equation in the language $\{\to, 1\}$  holds in every Heyting
algebra if and only if it holds in every implicative semilattice,
and this happens if and only if it holds in every Hilbert algebra.
In particular, we have that the equations
\begin{enumerate}
\item $(x\we y)\ra z \approx x \ra (y\ra z)$
\item $x\ra (y\we z) \approx (x\ra y)\we (x\ra z)$
\end{enumerate}
hold in every implicative semilattice. We highlight here the
properties of Hilbert algebras most relevant to the paper.

\begin{lem}
Let $H\in \Hil$ and $a,b,c\in H$. Then the following conditions
are satisfied:
\begin{enumerate}[\normalfont a)]
\item $a\ra a = 1$, \item $1\ra a = a$, \item $a\ra (b\ra c) =
b\ra (a \ra c)$, \item $a\ra (b\ra c) = (a\ra b)\ra (a\ra c)$,
\item if $a\leq b$, then $c\ra a \leq c\ra b$ and $b\ra c \leq a\ra
c$.
\end{enumerate}
\end{lem}

We denote the category of implicative semilattices (i.e., of the
implicative semilattices with their algebraic homomorphisms) by
$\IS$. Notice that if $(H,\we,\ra,1) \in \IS$,  then $(H,\ra,1)\in
\Hil$. \vspace{3pt}

Esakia considered in \cite{Esakia} the modalized Heyting calculus
$mHC$, which consists of an augmentation of the Heyting
propositional calculus for intuitionistic logic by a very special
modal operator. The algebraic models of this calculus are the
Heyting algebras augmented with a frontal operator. Let
$(H,\we,\vee,\ra,0,1)$ be a Heyting algebra. A map $\tau:H\ra H$
is said to be a \emph{frontal operator} if the following
conditions are satisfied for every $a,b \in H$:
\begin{description}
\item[(f1)] $\tau(a\wedge b) = \tau (a) \wedge \tau (b)$,
\item[(f2)] $a\leq \tau (a)$, \item[(f3)] $\tau (a)\leq b \vee
(b\ra a)$.
\end{description}
One of the main motivations to study frontal operators in Heyting
algebras stemmed from topological semantics in which $\tau$ is
interpreted as the co-derivative operator. In what follows we will
specify this point.

Let $X$ be a topological space. The derivative operator $\delta$
of $X$ is the map that sends every subset of $A$ to its set of
acummulation points. Its dual is the co-derivative operator $\tau$
defined by setting for every $A \subseteq X$, $\tau(A) =
\delta(A^{c})^{c}$. In \cite{Esakia} the elements of $\tau(A)$
were called the frontal points of $A$, that is,  $x \in X$ is a
frontal point of $A$ if and only if there is a neighborhood $U$ of
$x$ such that $U\subseteq A \cup \{x\}$. When $\tau$ is applied to
an open set provides an open set. In \cite{Esakia} Esakia also
showed that if $X$ is a topological space, then the co-derivative
operator restricted to $\mathcal{O}(X)$ is a frontal operator of
the Heyting algebra $\mathcal{O}(X)$.

We present now an interesting characterization of the co-derivatie
operator of the Heyting algebra of the open sets of a topological
space. Recall that  a point $x$ of a topological space $X$ is an
isolated  point of a set $A \subseteq X$ if there exists a
neighbourhood $U_x$ of $x$ such that $U_x \cap A = \{x\}$. We
denote the set of isolated points of $A$ by $A_a$.

\begin{prop}\label{pmot}
Let $X$ be a topological space and $\tau$ its co-derivative
operator. Then for every $U\in \mathcal{O}(X)$, $\tau(U) = U \cup
(U^c)_a$, that is, for every point $x \in X$, $x$ is a frontal
point of $U$ if and only if $x \in U$ or $x$ is an isolated point
of $U^{c}$.
\end{prop}

\begin{proof}
Let $x\in \tau (U)$. Suppose that $x\notin U$. We will see that $x
\in (U^c)_a$. Since $x\in \tau(U)$ then there exists a
neighborhood $V_x$ of $x$ such that $V_x\subseteq U\cup\{x\}$.
Thus, $V_x \cap U^c =\{x\}$. Hence, $x \in (U^c)_a$. Conversely,
let $x\in U\cup (U^c)_a$. If $x \in U$, since $U$ is open, $x$ is a
frontal point of $U$. If $x \in (U^{c})_a$, there exists a
neighborhood $V_x$ of $x$ such that $V_{x}\cap U^c=\{x\}$. Hence,
$V_{x} = V_{x}\cap(U\cup U^c) =(V_{x}\cap U)\cup \{x\}\subseteq
U\cup \{x\}.$ Therefore, $x$ is a frontal point of $U$.
\end{proof}

We can apply the proposition to the Heyting algebra of the upsets
of a poset. Let $(P,\leq)$ be a poset and $U\subseteq P$. We write
$U_M$ for the set of maximal elements of $U$ (note that this set
may be empty).

\begin{cor} \label{pmots}
Let $(P,\leq)$ be a poset and consider the topological space of
the upsets of $P$. Then $(U^c)_a = (U^c)_M$ for every $U\in P^+$.
Hence, the co-derivative frontal operator $\tau$ satisfies that
$\tau(U) = U \cup (U^c)_M$ for every $U \in P^{+}$.
\end{cor}

\begin{proof}
Straightforward computations show that if $x\in P$, then $U_x$ is a
neighborhood of $x$ in the topology of the upsets of $P$ if and
only if $[x) \subseteq U_x$. Using this fact we prove that
$(U^c)_a = (U^c)_M$ whenever $U\in P^+$. In order to do it, let
$U\in P^+$. Suppose that $x\in (U^c)_a$, so there exists a
neighborhood $U_x$ of $x$ such that $U_{x} \cap U^c = \{x\}$. Let
$x\leq y$ with $y\in U^c$. Then $y\in [x) \cap U^{c} \subseteq U_x
\cap U^c = \{x\}$, i.e., $y = x$. Therefore,  $x\in (U^c)_M$.
Conversely, let $x\in (U^c)_M$. Thus, $[x) \cap U^c = \{x\}$,
which implies that $x\in (U^c)_a$. Hence, $(U^c)_a = (U^c)_M$.
Finally, it follows from Proposition \ref{pmot} that  $\tau(U) = U
\cup (U^{c})_M$ for every $U \in P^{+}$.
\end{proof}

The following definition was introduced in \cite{CSM0} and
generalizes the definition of frontal operator given for Heyting
algebras to the Hilbert algebras setting.

\begin{defn}
Let $H\in \Hil$. We say that a map $\tau:H \to H$ is a
\emph{frontal operator} if it satisfies the following conditions
for every $a,b \in H$:
\begin{description}
\item[(i1)] $\tau(a\ra b) \leq \tau (a) \ra \tau (b)$, \item[(i2)]
$a\leq \tau(a)$, \item[(i3)] $\tau(a) \leq ((b\ra a)\ra b)\ra b$.
\end{description}
An algebra $(H,\tau)$ is a \emph{frontal Hilbert algebra} if $H$
is a Hilbert algebra and $\tau$ a frontal operator on it.
\end{defn}

We denote by $\FHil$ the algebraic category of frontal Hilbert
algebras (i.e., the morphisms are the algebra homomorphisms). In
every Hilbert algebra there exists at least one frontal operator
since the identity map meets the required conditions.

It turns out that a unary map $\tau$ on a Heyting algebra is a
frontal operator if and only if it satisfies the conditions
$\textbf{(i1)}$, $\textbf{(i2)}$,  and $\textbf{(i3)}$. This
explains why   frontal Hilbert algebras are a generalization of
frontal Heyting algebras.\footnote{In fact condition
$\textbf{(i1)}$ is equivalent to condition $\textbf{(f1)}$
(assuming that $\tau(1) = 1$) and condition $\textbf{(i3)}$
equivalent to condition \textbf{(f3)} (see \cite{CSM0,Esakia}).}

Let $A$ be an algebra. An $n$-ary function $f: A^{n} \ra A$  is
said to be \emph{compatible with a congruence} $\theta$ of $A$ if
$(a_{i}, b_{i}) \in \theta$ with $i= 1,...,n$ implies
$(f(a_{1},...,a_{n}), f(b_{1},...,b_{n}))\in \theta$. And it is
said to be a \emph{compatible function} of $A$ provided it is
compatible with all the congruences of $A$. The simplest examples
of compatible functions in an algebra $A$ are the polynomial
functions.\footnote{The notion of polynomial used here is simply
that from universal algebra (see \cite{Pix}).} Frontal operators
on Heyting algebras are necessarily compatible functions as well
as the frontal operators on Hilbert algebras are, as it was proved in
\cite{CSM0}.

\begin{defn}
An algebra $(H,\tau)$ is an \emph{frontal implicative semilattice}
if $H$ is an implicative semilattice and $\tau$ is a frontal
operator of its Hilbert algebra reduct.
\end{defn}

We denote by $\FIS$ the algebraic category of frontal implicative
semilattices. The following result holds as in the case of Heyting
algebras, and its proof is part of the folklore of the subject.

\begin{lem} \label{inf}
Let $H\in \IS$ and $\tau$ a unary operator on $H$. Then $\tau$ is
a frontal operator if and only if $\tau$ satisfies
$\mathbf{(i2)}$, $\mathbf{(i3)}$ and $\tau(a\we b) = \tau(a) \we
\tau(b)$ for every $a, b \in H$.
\end{lem}

\subsection{An adjunction between $\Hil$ and $\IS$}

The forgetful functor from the category of implicative
semilattices to the category of Hilbert algebras that forgets the
meet operation has a left adjoint. This amounts to the existence
of the free implicative semilattice extension of any Hilbert
algebra. There are several ways to obtain such a left adjoint.
Explicit descriptions of such an adjoint functor and of the free
extension of a Hilbert algebra to an implicative semilattice are
obtained in \cite{CJ2} and \cite{CSM}.

For completeness of the exposition, we provide now a description
of an adjoint functor to the forgetful functor from $\IS$ to
$\Hil$, which will be used later to obtain our results. We give
complete proofs using only the minimum tools needed to obtain the
results, thus avoiding more general approaches such that those in
\cite{CJ2,CSM}.

We start with some preliminary definitions and results.

It is immediate that if $(P, \leq)$ is a poset and $U, V \in
P^{+}$, then the following condition is satisfied for every $x \in
P$:
\[
x \in U \Rightarrow Y \quad \Longleftrightarrow \quad (\forall y \in P)(x \leq y \text{ and } y \in U \Longrightarrow y \in V).
\]

In the general study of Hilbert algebras, the notion of
implicative filter plays an important role. Let $H\in \Hil$. A set
$F\subseteq H$ is said to be an \emph{implicative filter} if $1\in
F$ and $b\in F$ whenever $a\in F$ and $a\ra b \in F$ for all $a, b
\in H$. If $F$ is a proper subset of $H$, then we say that the
implicative filter $F$ is \emph{proper}. It is immediate that
every implicative filter of a Hilbert algebra is an upset w.r.t.\
the natural order. We denote by $\Fil(H)$ the set of all
implicative filters of $H$.

Note that the set of all implicative filters of $H$ is closed
under intersections of arbitrary families, therefore it is a
complete lattice under the inclusion order. Thus, we can speak of
the implicative filter generated by a set. Let  $H\in \Hil$ and $X
\subseteq H$. We denote the implicative filter generated by $X$,
i.e., the least filter of $H$ that contains the set $X$, by
$F(X)$. There is an explicit description for $F(X)$ (see
\cite[Lemma 2.3]{Bu}):
\begin{align*}
F(X)   =  & \ \{b \in H: b = 1 \text{ or } a_1\ra (a_2\ra \cdots (a_n \ra b) \ldots)=1\\ & \ \text{for some}\; a_{1}, \ldots,a_n\in X\}.
\end{align*}
The lattice $\Fil(H)$ is known to be distributive \cite{D}.

The next fact on implicative filters, proved in \cite[Theorem
3.2]{Cel}, will be used several times along the paper. Let
$f:H_1\ra H_2$ be a function between Hilbert algebras. Then the
following two conditions are equivalent:
\begin{enumerate}
\item $f(1) = 1$ and $f(a\ra b) \leq f(a)\ra f(b)$ for every $a,b
\in H_1$. \item $f^{-1}[F]$ is an implicative filter of $H_1$
whenever $F$ is an implicative filter of $H_2$.
\end{enumerate}
In particular (2) holds when $f$ is a homomorphism.
\vspace{1pt}

An implicative filter $F$ of a Hilbert algebra $H$ is
\textit{irreducible} if it is an irreducible element of the
lattice of the implicative filters of $H$, i.e., if $F$ is proper
and for any implicative filters $F_1, F_2$ of $H$ such that $F =
F_1 \cap F_2$ we have  $F = F_1$ or $F = F_2$. We denote by
$\X(H)$ the set of irreducible implicative filters of $H$, as well
as the poset we obtain ordering it by the inclusion relation.
\vspace{1pt}

For a proof of the following lemma see \cite{D}.

\begin{lem} \label{SC}
Let $H\in \Hil$ and $F\in \Fil(H)$. Then $F\in \X(H)$ if and only
if for every $a,b\notin F$ there exists $c\notin F$ such that
$a\leq c$ and $b\leq c$.
\end{lem}

Let $H\in \Hil$. A set $I \subseteq H$ is an \emph{order-ideal} if
$I$ is an order-ideal w.r.t.\ the natural order. We denote by
$\Id(H)$  the set of order-ideals of $H$ w.r.t.\ the natural order.
The following lemma is \cite[Theorem 2.6]{Cel}.

\begin{lem}\label {tfp}
Let $H\in \Hil$. Let $F\in \Fil(H)$ and $I \in \Id(H)$ be such
that $F\cap I = \emptyset$. Then there exists $P\in \X(H)$ such
that $F\subseteq P$ and $P\cap I = \emptyset$.
\end{lem}

The following known results follow from  Lemma \ref{tfp}.

\begin{cor}\label{tfpc1}\label{tfpc2}\label{tfpc3}
Let $H\in \Hil$.
\begin{enumerate}[\normalfont 1)]
\item If  $F \in \Fil(H)$ and $a \notin F$, then  there exists
$P\in \X(H)$ such that $F\subseteq P$ and $a\notin P$. \item If
$a,b\in H$ are such that $a\nleq b$, then there exists $P\in
\X(H)$ such that $a\in P$ and $b\notin P$. \item If $F\in \Fil(H)$
and $a,b\in H$, then $a\ra b \notin F$ if and only if there exists
$P\in \X(H)$ such that $F\subseteq P$, $a\in P$ and $b\notin P$.
\end{enumerate}
\end{cor}

The next lemma was proved in  \cite[Section 3, Lemma 16]{CSM0}. We
give the proof here for the sake of completeness.

\begin{lem}\label{irr-tau}
Let $(H,\tau)$ be a frontal Hilbert algebra and $P \in \X(H)$.
Then for every $a, b \in H$, if $\tau(a) \in P$ and $b \not \in
P$, then $b \to a \in P$.
\end{lem}

\begin{proof}
Let $P\in \X(H)$, $\tau(a)\in P$, and $b\notin P$. Suppose that $b
\to a \not \in P$. Then by Lemma \ref{SC} there exists $c\notin P$
such that $b\leq  c $ and $b\ra a \leq c$. Thus, $b \to c = 1$ and
$(b\ra a) \ra c = 1$. Since $(b\ra c)\ra ((b\ra a) \ra c)\ra
((c\ra a)\ra c)) = 1$, then $(c\ra a)\ra c = 1$. From the fact
that $\tau$ is a frontal operator we have $\tau(a)\leq ((c\ra
a)\ra c)\ra c$, i.e., $\tau(a)\leq c$. Taking into account that
$\tau(a)\in P$ we conclude that $c\in P$, which is a
contradiction.
\end{proof}

Let $H\in \Hil$. We consider the poset $\X(H) = (\X(H),
\subseteq)$ and the complete lattice of its upsets $\X(H)^{+}$. We
define the map $\varphi_H: H \to \X(H)^{+}$ by setting for every
$a\in H$
\begin{equation} \label{var}
\varphi_H(a): = \{P\in \X(H): a\in P\}.
\end{equation}
This map is well-defined since it is immediate that $\varphi_H(a)$
is an upset of $\X(H)$. When the algebra $H$ is clear from the
context we will use $\varphi$ instead of $\varphi_H$. Corollary
\ref{tfpc2} implies that $\varphi$ is a one-to-one map. Hence,
$\varphi$  is an order embedding from the poset $(H,\leq)$, where
$\leq$ is the natural order of $H$, to the complete lattice poset
$(\X(H)^{+}, \subseteq)$; therefore, $((\X(H)^{+}, \subseteq),
\varphi)$ is a completion of $(H,\leq)$.

The operation $\Rightarrow$ on $\X(H)^{+}$, defined by condition
(\ref{imp}) in Section~\ref{Sec:Prelim}, is such that for every
$a, b \in H$
\[
\varphi(a) \Rightarrow \varphi(b) = \varphi(a \to b).
\]
Indeed, if $P \in \varphi(a \to b)$, $P \subseteq Q \in \X(H)$ and
$a \in Q$, then, since $Q$ is an implicative filter,  we have $b
\in Q$. It follows that $\varphi(a \to b) \subseteq \varphi(a)
\Rightarrow \varphi(b)$. Conversely, if $P \not \in \varphi(a \to
b)$, using Corollary \ref{tfpc3} there is $Q \in \X(H)$ such that
$P \subseteq Q$, $a \in Q$, and $b \not \in Q$. It follows then
that $P \not \in \varphi(a) \Rightarrow \varphi(b)$. Hence, we
obtain the other inclusion. Moreover $\varphi(1) = \X(H)$ and
therefore, the map  $\varphi$ is an embedding from $H$ to the
Hilbert algebra $(\varphi[H], \Rightarrow, \X(H))$, which is a
subalgebra of the $\{\to, 1\}$-reduct of the Heyting algebra of
the upsets of $\X(H)$.

\begin{rem}
The completion $((\X(H)^{+}, \subseteq), \varphi_H)$ of the poset
$(H, \leq)$ is  a $\Delta_1$-comple\-tion in the sense of
\cite{GeJaPa}. In \cite{Go} it is proved that it is indeed the
$(\Fil(H), \Id(H))$-completion of $(H, \leq)$ and that the
operation $\Rightarrow$ is the $\pi$-extension of the operation
$\to$ of $H$ to $\X(H)^{+}$.
\end{rem}

Using that $\X(H)^{+}$ is a Heyting algebra and the fact that
$\varphi$ is an embedding from $H$ to  $(\varphi[H], \Rightarrow,
\X(H))$, it is easy to see that for every $a_1, \ldots, a_n, b \in
H$,
\[
\varphi(a_1) \cap \cdots \cap \varphi(a_n) \Rightarrow \varphi(b) = \varphi(a_1 \to (\ldots (a_n \to b)\ldots)).
\]

Let $H\in \Hil$. We consider the bounded semilattice
$(\X(H)^{+}, \cap, \X(H))$ and the subalgebra generated by
$\varphi[H]$, which is, of course,  a bounded semilattice. We denote it, as
well as its domain, by $\langle \varphi[H]\rangle$. Since $\X(H) =
\varphi(1) \in \varphi[H]$, the elements of $\langle
\varphi[H]\rangle$ are the sets of the form
\[
U = \varphi(a_1) \cap \cdots \cap \varphi(a_n)
\]
for some $a_1, \ldots, a_n \in H$.

\begin{prop}\label{closure-L(H)-under-Rightarrow}
For every Hilbert algebra $H$ the set  $\langle \varphi[H]\rangle$
is closed under the operation $\Rightarrow$ of $\X(H)^{+}$.
\end{prop}

\begin{proof}
Let $U, V \in \langle \varphi[H]\rangle$. Assume that $U =
\varphi(a_1) \cap \cdots \cap \varphi(a_n)$ and $V =  \varphi(b_1)
\cap \cdots \cap \varphi(b_m)$, where $a_1, \ldots, a_n, b_1,
\ldots, b_m \in H$. Then
\[
U \Rightarrow V = \varphi(a_1) \cap \cdots \cap \varphi(a_n) \Rightarrow \varphi(b_1) \cap \cdots \cap \varphi(b_m).
\]
Using that $\X(H)^{+}$ is a Heyting algebra, we have
\[
U \Rightarrow V = \bigcap_{1 \leq i \leq m} \varphi(a_1) \cap \cdots \cap \varphi(a_n) \Rightarrow \varphi(b_i).
\]
Now, for every $1 \leq i \leq m$ we have
\[
\varphi(a_1) \cap \cdots \cap \varphi(a_n) \Rightarrow \varphi(b_i) = \varphi(a_1 \to (\ldots (a_n \to b_i)\ldots)).
\]
It follows that $U \Rightarrow V \in  \langle \varphi[H]\rangle$.
\end{proof}

The proposition implies that the algebra
\[
\FC(H) := (\langle \varphi[H]\rangle, \cap, \Rightarrow, \X(H))
\]
is an implicative semilattice.

\begin{lem}\label{exten-hom-Hilbert-1}
Let $h: H_1 \to H_2$ be a homomorphism of Hilbert algebras.
If $a_1, \ldots, a_n,$ $ b_1, \ldots, b_m \in H$ are such that
\[
\varphi_{H_1}(a_1) \cap \cdots \cap \varphi_{H_1}(a_n) \subseteq
\varphi_{H_1}(b_1) \cap \cdots \cap \varphi_{H_1}(b_m),
\]
 then
\[
\varphi_{H_2}(h(a_1)) \cap \cdots \cap \varphi_{H_2}(h(a_n))
\subseteq \varphi_{H_2}(h(b_1)) \cap \cdots \cap
\varphi_{H_2}(h(b_m)).
\]
\end{lem}
\begin{proof}
Let $a_1, \ldots, a_n, b_1, \ldots, b_m \in H$ with
$\varphi_{H_1}(a_1) \cap \cdots \cap \varphi_{H_1}(a_n) \subseteq
\varphi_{H_1}(b_1) \cap \cdots \cap \varphi_{H_1}(b_m)$. Suppose
that $P \in \varphi_{H_2}(h(a_1)) \cap \cdots \cap
\varphi_{H_2}(h(a_n))$. Then $a_1, \ldots, a_n \in h^{-1}[P]$.
Suppose that there is $b_i$ with $1 \leq i \leq m$ such that
$h(b_i) \not \in P$. Then $b_i \not \in h^{-1}[P]$. Thus, there
exists $Q \in \X(H)$ such that $b_i \not \in Q$ and $h^{-1}[P]
\subseteq Q$. It follows that  $Q \in \varphi_{H_1}(a_1) \cap
\cdots \cap \varphi_{H_1}(a_n)$, and the assumption implies that
$b_i \in Q$, which is a contradiction. Therefore, $P \in
\varphi_{H_2}(h(b_1)) \cap \cdots \cap \varphi_{H_2}(h(b_m))$.
This conludes the proof.
\end{proof}

\begin{prop}\label{unique -hom-extension}
Let $h: H_1 \to H_2$ be a homomorphism of Hilbert algebras. Then
there exists a unique homomorphism $\hat{h}: \FC(H_1) \to
\FC(H_2)$ such that $\varphi_{H_2} \circ h = \hat{h} \circ
\varphi_{H_1}$, i.e., that makes the following diagram to commute:
\[
 \xymatrix{
   H_1 \ar[rr]^{\varphi_{H_1}} \ar[d]_{h} & & L(H_1)  \ar[d]^{\hat{h}}\\
   H_2 \ar[rr]^{\varphi_{H_2}} & &  L(H_2). \
   }
\]
\end{prop}

\begin{proof}
First we show that if such a homomorphism exists, it is unique.
Suppose that $f: \FC(H_1) \to \FC(H_2)$ is a homomorphism such
that $\varphi_{H_2} \circ h = f \circ \varphi_{H_1}$. Let $U =
\varphi_{H_1}(a_1) \cap \cdots \cap \varphi_{H_1}(a_n) \in \langle
\varphi_{H_1}[H_1]\rangle$ with $a_1, \ldots, a_n \in H_1$. Then
$f(U) = f(\varphi_{H_1}(a_1)) \cap \cdots \cap
f(\varphi_{H_1}(a_n)) = \varphi_{H_2}(h(a_1)) \cap \cdots \cap
\varphi_{H_2}(h(a_n))$. This implies that  if $f_1, f_2: \FC(H_1)
\to \FC(H_2)$ are homomorphism such that $\varphi_{H_2} \circ h =
f_1 \circ \varphi_{H_1}$ and $\varphi_{H_2} \circ h = f_2 \circ
\varphi_{H_1}$, it follows that for every $U \in \langle
\varphi_{H_1}[H_1]\rangle$, $f_1(U) = f_2(U)$.

Now we prove the existence. We define $\hat{h} : \FC(H_1) \to
\FC(H_2)$ by setting for every $U \in \langle
\varphi_{H_1}[H_1]\rangle$:
\[
\hat{h}(U) = \varphi_{H_2}(h(a_1)) \cap \cdots \cap
\varphi_{H_2}(h(a_n)),
\]
where $a_1 \ldots, a_n \in H_1$ are such that $U =
\varphi_{H_1}(a_1) \cap \cdots \cap \varphi_{H_1}(a_n)$. Lemma
\ref{exten-hom-Hilbert-1} implies that  the map $\hat{h}$ is well
defined. Note that in particular $\hat{h}(\varphi_{H_1}(a)) =
\varphi_{H_2}(h(a))$. It is immediate to see that  $\varphi_{H_2}
\circ h = \hat{h} \circ \varphi_{H_1}$, that for all $U, V \in
\X(H)$ it holds that $\hat{h}(U \cap V) = \hat{h}(U) \cap
\hat{h}(V)$ and that $\hat{h}(\X(H_1)) = \X(H_2)$. It remains to
show that $\hat{h}(U \Rightarrow V) = \hat{h}(U) \Rightarrow
\hat{h}(V)$ for all $U, V \in \X(H)$. Suppose that $U =
\varphi_{H_1}(a_1) \cap \cdots \cap \varphi_{H_1}(a_n)$ and $V =
\varphi_{H_1}(b_1) \cap \cdots \cap \varphi_{H_1}(b_m)$. Reasoning
as in the proof of Proposition \ref{closure-L(H)-under-Rightarrow}
we have
\[
U \Rightarrow V = \bigcap_{1 \leq i \leq m} \varphi_{H_1}(a_1 \to (\ldots (a_n \to b_i)\ldots)).
\]
Hence, $\hat{h}(U \Rightarrow V) = \bigcap_{1 \leq i \leq m}\varphi_{H_2}(h(a_1 \to (\ldots (a_n \to b_i)\ldots)))$. It easily follows that
\[
\hat{h}(U \Rightarrow V) = \bigcap_{1 \leq i \leq m} \varphi_{H_2}(h(a_1)) \cap \cdots \cap \varphi_{H_2}(h(a_n)) \Rightarrow \varphi_{H_2}(h(b_i)).
\]
The last expression is equal to $\hat{h}(U) \Rightarrow
\hat{h}(V)$. We conclude that $\hat{h}$ is the desired
homomorphism.
\end{proof}

Using the results above the next proposition easily follows.

\begin{prop} \label{propfun}
The assignments $H\mapsto \FC(H)$ and $h\mapsto \hat{h}$ define a
functor $\f:\Hil \ra \IS$.
\end{prop}

Recall that if $H\in \IS$, a subset $F\subseteq H$ is said to be a
\emph{filter} if it is an upset,  $1 \in
F$, and $a\we b \in F$ whenever $a,b \in F$. It
is part of the folklore that if $H\in \IS$, then the set of
implicative filters of $H$ is equal to the set of filters of $H$.
\vspace{1pt}

Let $\U$ be the forgetful functor from $\IS$ to $\Hil$; namely,
$\U$ sends every implicative semilattice to its Hilbert algebra
reduct and the homomorphisms accordingly.

\begin{prop}\label{pu}
Let $H$ be a Hilbert algebra and let $A$ be an implicative
semilattice. Consider the $\{\to, 1\}$-fragment $\U(A)$ of $A$ and
a homomorphism $h: H \to  \mathrm{U}(A)$. Then, there exists a
unique homomorphism $\overline{h}:\FC(H) \ra A$ such that $h =
\overline{h}  \circ \varphi_H$.
\end{prop}

\begin{proof}
Straightforward computations show that the map $\varphi_{\U(A)}:
\U(A) \to \FC(\U(A))$, that we abbreviate in this proof as
$\varphi_A$, is in fact an isomorphism between $A$ and
$\FC(\U(A))$, because $\varphi_{A}(a) \cap \varphi_{A}(b) =
\varphi_{A}(a \wedge b)$ for every $a, b \in A$. By Proposition
\ref{unique -hom-extension} we have that there exists a unique
homomorphism $\hat{h}: \FC(H) \to \FC(U(A))$ such that
$\hat{h}\circ \varphi_H = \varphi_{A} \circ h$. Let $\overline{h}:
L(H) \ra A$ be the map $\varphi_{A} ^{-1} \circ \hat{h}$. Then it
follows that $h = \overline{h}  \circ \varphi_H$. This proves the
existence. To prove uniqueness suppose the $f_1, f_2: \FC(H) \ra
A$ are such that  $h = f_1 \circ \varphi_H$ and $h = f_2  \circ
\varphi_H$. Then $\varphi_A \circ h = (\varphi_A \circ f_1) \circ
\varphi_H$ and $\varphi_A \circ h = (\varphi_A \circ f_2) \circ
\varphi_H$. Therefore, Proposition \ref{unique -hom-extension}
also implies that $\varphi_A \circ f_1 = \varphi_A \circ f_2$.
Since $\varphi_A$ is one-to-one, it follows that $f_ 1= f_2$.
\end{proof}

Let $\mathrm{I_{\Hil}}$ be the identity functor in $\Hil$. From
Proposition \ref{unique -hom-extension} follows that the morphisms
$\varphi_H: H \to \FC(H)$ establish a natural transformation from
$\mathrm{I_{\Hil}}$ to the functor $\U \circ \f$. Then using
Proposition \ref{pu} we obtain the following result.

\begin{thm} \label{pteo}
The functor $\f: \Hil \ra \IS$ is left adjoint to $\U$.
\end{thm}

An algebra $(H,\ra,0,1)$ of type $(2,0,0)$ is a \emph{bounded
Hilbert algebra} if $(H,\ra,1)$ is a Hilbert algebra and $0\leq a$
for every $a\in A$. We write $\Hil_0$ for the algebraic category
of bounded Hilbert algebras. An algebra $(H,\we,\ra,0,1)$ of type
$(2,2,0,0)$ is a \emph{bounded implicative semilattice} if
$(H,\we,\ra,1)$ is an implicative semilattice and $0\leq a$ for
every $a\in H$. We write $\IS_0$ for the algebraic category of
bounded implicative semilattices. Note that  if $H$ is a bounded
Hilbert algebra, then $\varphi_{H}(0) = \emptyset$ and therefore
the bottom element of the Heyting algebra $\X(H)^{+}$ belongs to
the implicative semilattice $\langle \varphi_{H}[H]\rangle$ and
hence it is a bounded implicative semilattice. We define the
functors $(\ )^{\IS}: \Hil_0 \ra \IS_0$ and $\U$ similarly to
those of Theorem \ref{pteo}. Straightforward modifications of
propositions \ref{unique -hom-extension} and \ref{pu} and their
proofs show the following result.

\begin{cor} \label{pteob}
The functor $(\ )^{\IS}: \Hil_0 \ra \IS_0$ is left adjoint to $\U$.
\end{cor}

\section{An adjunction between $\FHil$ and $\FIS$}\label{FHil}

A frontal operator in a Hilbert algebra resembles a modal operator
$\Box$ in a Boolean algebra or in a distributive lattice in may
respects. Let $(H,\tau)$ be a frontal Hilbert algebra. We can
extend $\tau$ to the algebra $\X(H)^{+}$ is a similar way that in
a normal modal algebra we extend $\Box$ to the powerset algebra of
the ultrafilters or in a distributive lattice with a normal $\Box$
we extend it to the distributive lattice of the upsets of the
poset of its prime filters. We do it in the next definition.

\begin{defn}
Let $(H,\tau)\in \FHil$. We define the map $\tp:\X(H)^+ \ra
\X(H)^+$ in the following way:
\[
P\in \tp(U) \Longleftrightarrow\;(\forall Q\in \X(H))(
\tau^{-1}[P]\subseteq Q \Longrightarrow Q\in U).
\]
Note that the map is well-defined since from the definition
follows that $\tp(U)$ is an upset.
\end{defn}

The restriction of $\tp$ to $\varphi[H]$ is in fact (modulo
isomorphism) $\tau$ as shown in the next lemma.

\begin{lem} \label{laux}
Let $(H,\tau)$ be a frontal Hilbert algebra. Then for every $a \in
H$
\[
\varphi(\tau(a)) = \tp(\varphi(a)).
\]
\end{lem}

\begin{proof}
Let $P\in \varphi(\tau(a))$, i.e., $\tau(a) \in P$. Let $Q\in
\X(H)$ such that $\tau^{-1}[P] \subseteq Q$. Then $a \in Q$. Thus
it follows that $\varphi(\tau(a)) \subseteq \tp(\varphi(a))$. To
prove the other inclusion, let $P \in \tp(\varphi(a))$ and assume
that $\tau(a) \not \in P$. Thus, $a \not \in \tau^{-1}[P]$. Since
it holds that $\tau(1) = 1$ and $\tau(a\ra b)\leq \tau(a) \ra
\tau(b)$ for every $a,b\in H$, it follows that $\tau^{-1}[P]$ is
an implicative filter. Therefore, there is $Q \in \X(H)$ such that
$\tau^{-1}[P] \subseteq Q$ and $a \not \in Q$, which is a
contradiction with the fact that $P \in \tp(\varphi(a))$.
\end{proof}

In what follows we will prove that if $(H,\tau)\in \FHil$, then
$(\X(H)^+,\tp)$ is a frontal Heyting algebra (in particular, the
appropriate reducts are a frontal implicative semilattice and a
frontal Hilbert algebra).

\begin{prop}\label{pp}
For every frontal Hilbert algebra $(H,\tau)$, the algebra
$(\X(H)^+,\tp)$ is a frontal Heyting algebra and $\varphi$ is an
embedding of frontal Hilbert algebras from $(H,\tau)$ to
$(\X(H)^+,\tp)$.
\end{prop}

\begin{proof}
First of all note that for every $P \in \X(H)^+$ we have that
$P\subseteq \tau^{-1}[P]$. This holds because $a\leq \tau(a)$ for
every $a\in H$.

Let $U\in \X(H)^+$. We show that $U\subseteq \tp(U)$. To this end,
let $P\notin \tp(U)$. Thus, there exists $Q\in \X(H)$ such that
$\tau^{-1}[P]\subseteq Q$ and $Q\notin U$. Since $P\subseteq
\tau^{-1}[P]$, then $P\subseteq Q$. But $Q\notin U$, so $P\notin
U$.

Straightforward computations based on the definition of $\tp$ show
that for all $U,V\in \X(H)^+$, $\tp(U\cap V) = \tp(U) \cap \tp(V)$.

 Finally we will see that $\tp(U) \subseteq V
\cup (V\Ra U)$, for every $U,V\in \X(H)^+$. Suppose that there
exists $P\in \tp(U)$ such that $P\notin V \cup (V\Ra U)$. Hence,
$P\notin V$ and there exists $Q\in \X(H)$ such that $P\subseteq
Q$, $Q\in V$ and $Q\notin U$. Since $P\in \tp(U)$ and $Q\notin U$
then $\tau^{-1}[P]\nsubseteq Q$, which implies that there exists
$a\in H$ such that $\tau(a)\in P$ and $a\notin Q$. Also notice
that $Q\nsubseteq P$ because $P\notin V$, $Q\in V$ and $P\subseteq
Q$. Hence, there exists $b\in H$ such that $b\in Q$ and $b\notin
P$. Since $\tau(a)\in P$ and $b\notin P$ then it follows from
Lemma \ref{irr-tau} that $b\ra a \in P$, so $b\ra a \in Q$. Since
$b\in Q$ then $a \in Q$, which is a contradiction. We conclude
that $\tp(U) \subseteq V \cup (V\Ra U)$.

Therefore we have shown that $(\X(H)^+,\tp)$ is a frontal Heyting
algebra.
\end{proof}

Let $(H,\tau)\in \FHil$ and $U\in \FC(H)$. Then there exist
$a_1,\ldots,a_n\in H$ such that $U = \varphi(a_1)\cap\cdots
\varphi(a_n)$. Since $\tp$ is a frontal operator on the Heyting
algebra $\X(H)^+$,  $\tp(U) = \varphi(\tau(a_1))\cap\cdots
\cap \varphi(\tau(a_n)) \in \FC(H)$. It follows from Proposition
\ref{pp} that the restriction of $\tp$ to $\FC(H)$ is a function
$\tp:\FC(H) \ra \FC(H)$ which is a frontal operator on $\FC(H)$.
Taking into account Proposition \ref{propfun} we obtain the
following result.

\begin{cor}\label{paf1}
Let $(H,\tau)\in \FHil$. Then $(\FC(H),\tp)\in \FIS$.
\end{cor}

If $h:(H_1,\tau_1)\ra (H_2,\tau_2)$ is a morphism in $\FHil$, then
$h:H_1\ra H_2$ is a morphism in $\Hil$. Tehrefore, it follows from
Proposition \ref{propfun}  that $h^{\IS}:\FC(H_1)\ra \FC(H_2)$ is a
morphism in $\IS$. In the sequel we prove that
$h^{\IS}:(\FC(H_1),\tup)\ra (\FC(H_2),\tdp)$ is also a morphism in
$\FIS$.

\begin{lem} \label{paf2}
Let $h:(H_1,\tau_1)\ra (H_2,\tau_2)$ be a morphism in $\FHil$.
Then the function $\hIS:(\FC(H_1),\tup)\ra (\FC(H_2),\tdp)$ is a
morphism in $\FIS$.
\end{lem}

\begin{proof}
We only need to prove that $\hIS(\tup(U)) = \tdp(\hIS(U))$ for
every $U\in \FC(H_1)$. Let $U\in \FC(H_1)$, so there exist
$a_1,\ldots,a_n\in H_1$ such that $U = \bigcap_{i=1}^{n}
\varphi_{H_1}(a_i)$. Taking into account that $\hIS$ is a morphism
in $\IS$ and that $\hIS(\varphi_{H_1}(a)) = \varphi_{H_2}(h(a))$,
we obtain that
\[
\begin{array}
[c]{lllll}
\hIS(\tup(U))& = & \hIS(\bigcap_{i=1}^{n} \varphi_{H_1}(\tau_{1}(a_i))) &  & \\
&        =    & \bigcap_{i=1}^{n} \hIS(\varphi_{H_1}(\tau_{1}(a_i)))    &  & \\
&        =    & \bigcap_{i=1}^{n} \varphi_{H_2}(h(\tau_{1}(a_i)))   &  & \\
&        =    & \bigcap_{i=1}^{n} \varphi_{H_2}(\tau_{2}(h(a_i)))   &  & \\
&        =    & \tdp(\bigcap_{i=1}^{n} \varphi_{H_2}(h(a_i)))     &  & \\
&        =    & \tdp(\bigcap_{i=1}^{n} \hIS(\varphi_{H_1}(a_i)))     &  & \\
&        =    & \tdp(\hIS(\bigcap_{i=1}^{n} \varphi_{H_1}(a_i)))       &  & \\
&        =    & \tdp(\hIS(U)).& &
\end{array}
\]
Therefore we obtain that $\hIS(\tup(U)) =
\tdp(\fIS(U))$, which was our aim.
\end{proof}

Then we obtain the next proposition.

\begin{prop}\label{ps5}
The functor $\f:\Hil \ra \IS$ can be extended to a functor
$\fro: \FHil \ra \FIS$.
\end{prop}

We also write $\U$ for the forgetful functor from $\FIS$ to
$\FHil$. The  next proposition follows from
Proposition \ref{pteo}, Proposition \ref{ps5}, and the fact that if
$(H,\tau)\in \FIS$, then $\varphi:(H,\tau) \ra (\FC(H),\tp)$ is an
embedding.

\begin{thm} \label{HiltoIS}
The functor $\fro: \FHil \ra \FIS$ is left adjoint to $\U$.
\end{thm}

Let $\FHil_0$ be the algebraic category of frontal bounded Hilbert
algebras and $\FIS_0$ the algebraic category of frontal bounded
implicative semilattices. We define the functor $\U$ from $\FIS_0$
to $\FHil_0$ similarly to that of Theorem \ref{HiltoIS}.

The following two corollaries are consequence of Corollary
\ref{pteob} together with similar ideas to those used to obtain
Proposition \ref{ps5} and Theorem \ref{HiltoIS}.

\begin{cor} \label{cor0}
The functor $\fro: \FHil \ra \FIS$ can be extended to a functor
$\fro: \FHil_0 \ra \FIS_0$.
\end{cor}

\begin{cor} \label{cor1}
The functor $\fro: \FHil_0 \ra \FIS_0$ is left adjoint to $\U$.
\end{cor}

\section{An adjunction between $\gaHil$ and $\gaIS$}\label{gaHil}

In \cite[Example 3.1]{Caici} Caicedo and Cignoli introduced and
studied some implicit compatible operations of Heyting algebras.
One of them was called the $\gamma$ operation, which is a frontal
operator that satisfies the extra conditions that we proceed to
list.

A frontal operator $\tau$ on a Heyting algebra $H$ will be called
a $\gamma$-\emph{operator} if it satisfies :
\begin{enumerate}
\item $\neg \tau(0) = 0$ \item $\tau(a) \leq a \vee \tau(0)$, for every
$a\in H$.
\end{enumerate}

In \cite[Proposition 2.4]{CSSM} it is shown that a unary map
$\gamma$ on a  Heyting algebra $H$ is a $\gamma$-operator if and
only if it satisfies  the following conditions
we find in \cite[Example 3.1]{Caici}:
\begin{enumerate}
\item $\neg \gamma(0) = 0$ \item $\gamma(0)\leq a \vee \neg a$,
\item $\gamma(a) = a \vee \gamma(0)$.
\end{enumerate}
for every $a\in H$.

Note that the condition $\gamma(a) = a \vee \gamma(0)$ can be
replaced by  $\gamma(a) \leq a \vee \gamma(0)$. If a
$\gamma$-operator $\tau$ exists, it is unique since it is
characterized by the condition
\[
\tau(a) =\; \text{min}\;\{b: \neg b \vee a\leq b\},
\]
for every $a \in H$, as it was proved in \cite{CSSM}. For this
reason, when a $\gamma$-operator exists, it is denoted by
$\gamma$. The operator $\gamma$ exists in every finite Heyting
algebra (see \cite{Caici}), however there are Heyting algebras
(necessarily infinite) where there is no $\gamma$-operator.

In \cite{CSM0} the notion of $\gamma$-operator was generalized to
the framework of bounded Hilbert algebras. As in Heyting algebras,
if  $H\in \Hil_0$ and $a\in H$ we define $\neg a: = a\ra 0$. Let
us say that a frontal operator $\tau$ on a bounded Hilbert algebra
$H$ is a $\gamma$-\emph{operator} if it satisfies for every $a,b
\in H$ the following conditions:
\begin{description}
\item[(g4)] $\neg \tau(a)\leq \tau(a)$. \item[(g5)] $\tau(a) \leq
(a\ra b) \ra ((\neg b \ra b) \ra b))$.
\end{description}

In \cite[Section 3, Proposition 7]{CSM0} it is proved that a
unary map $\tau$ on a bounded Hilbert algebra $H$ is a
$\gamma$-operator if and only if for every $a \in H$
\[
\tau(a) =\;\text{min}\;\{b\in H: \neg b\leq b\;\text{and}\;a\leq
b\}.
\]
Thus, there is at most one $\gamma$-operator on a bounded Hilbert
algebra. But $\gamma$-operators may not exist. In \cite[Section 3,
Example 14]{CSM0}, in contrast with the case of Heyting algebras,
we find examples of finite bounded Hilbert algebras that lack the
$\gamma$-operator. When a $\gamma$-operator exists, it will be
denoted, as we did for Heyting algebras, by $\gamma$.

Note that if $H\in \Hil_0$ and $b\in H$ then $\neg b\leq b$ if and
only if $\neg b = 0$. Then condition $\textbf{(g4)}$ can be
replaced by  $\neg\gamma(a) = 0$.

\begin{lem}
Let $H\in \Hil_0$ and $f:H\ra H$ a function which is monotone
w.r.t. the natural order. Then $\neg f(a) = 0$ for every $a\in H$
if and only if $\neg f(0) = 0$.
\end{lem}

\begin{proof}
Suppose that $\neg f(0) = 0$ and let $a\in H$. Since $0\leq a$
then $f(0)\leq f(a)$, so $\neg f(a) \leq \neg f(0) = 0$. Hence,
$\neg f(a) = 0$.
\end{proof}

The next corollary easily follows.

\begin{cor} \label{corze}
Let $H\in \Hil_0$. A frontal operator $\tau$ on $H$  is a
$\gamma$-operator if and only if $\neg \tau(0) = 0$ and condition
$\mathbf{(g5)}$ holds.
\end{cor}

We write $\gaHil$ for the algebraic category whose objects are
algebras $(H,\gamma)$, where $H\in \Hil_0$ and $\gamma$ is a
$\gamma$-operator. In a similar way we define the category
$\gaIS$.

Let $H\in \Hil_0$. For every $a\in H$ we define the set
\[
\gamma_{a} := \{b\in H:\neg b \leq b\;\text{and}\;a\leq b\}.
\]

\begin{prop}
Let $H \in \IS_0$. For every $a\in H$ the set $\gamma_{a}$ is a
filter. Moreover, if $H$ is finite then there exists the minimum
of $\gamma_a$ for every $a\in H$, i.e., $H$ has a
$\gamma$-operator.
\end{prop}

\begin{proof}
Let $a\in H$. It is immediate that $1\in \gamma_a$ and that
$\gamma_a$ is an upset. In what follows we will show that if
$b,c\in \gamma_a$ then $b\we c\in \gamma_a$. Let $b,c\in
\gamma_a$, i.e., $\neg b\leq b$, $\neg c \leq c$, $a\leq b$ and
$a\leq c$. Thus, $\neg b = \neg c = 0$ and $a\leq b\we c$. Then
$(b\we c)\ra 0 = b\ra (c\ra 0) = b\ra 0 = 0$. Hence, $b\we c\in
\gamma_a$.
\end{proof}

The following is \cite[Section 4, Lemma 15]{CSM0}.

\begin{lem}\label{gr}
Let $(H,\gamma) \in \gaHil$. Then $\varphi(\gamma(a)) = \varphi(a)
\cup (\X(H))_M$ for every $a\in H$.
\end{lem}

Let $(P,\leq)$ be a poset. Note that if $U\in P^+$,  then $U \cup
X_M\in P^+$.

\begin{lem} \label{grb}
Let $(H,\gamma) \in \gaHil$. Then $(\X(H)^+, \gap) \in \gaIS$.
Moreover, $\gap$ takes the form $\gamma^{\pi}(U) = U \cup
(\X(H))_M$.
\end{lem}

\begin{proof}
We already know that $\gap$ is a frontal operator. To prove that
it is a $\gamma$-operator we first show that $\gap(\emptyset) =
(\X(H))_M$. Let $P\in (\X(H))_M$. By Corollary \ref{corze} we have
that $\neg \gamma(0)\notin P$, so there exists $Q\in \X(H)$ such
that $P\subseteq Q$ and $\gamma(0)\in Q$. Thus, $P = Q$. Since
$\gamma(0)\in Q$, then $\gamma(0)\in P$. Therefore, since $\gamma$
is monotone, for every $a \in H$, $\gamma(a) \in P$, i.e.,
$\gamma^{-1}[P] = H$. Then $P\in \gap(\emptyset)$. Thus,
$(\X(H))_M \subseteq \gap(\emptyset)$. Conversely, assume that
$P\in \gap(\emptyset)$ and that $P\notin (\X(H))_M$. It follows
from Lemma \ref{gr} that $\varphi(\gamma(0)) = (\X(H))_M$, so
$\gamma(0)\notin P$, i.e., $0\notin \gamma^{-1}[P]$. Hence, there
exists $Q\in \X(H)$ such that $\gamma^{-1}[P]\subseteq Q$, which
contradicts the fact that $P\in \gap(\emptyset)$. Then
$\gap(\emptyset) \subseteq (\X(H))_M$. Hence, $\gap(\emptyset) =
(\X(H))_M$.

Next we see that $\neg \gap(\emptyset) = \emptyset$. Since
$(\X(H))_M = \gap(\emptyset)$ then $\neg \gap(\emptyset) = \neg
(\X(H))_M$. But $\neg (\X(H)_M) = ((\X(H))_M]^c$. Since
$\varphi(0) = \emptyset$, it follows from Lemma \ref{gr} that
$((\X(H))_M] = \X(H)$. Then $\neg \gap(\emptyset) = \emptyset$.

Now we show that $\gap(U) \subseteq U \cup \gap(\emptyset)$ for
every $U\in \X(H)^+$. Let $U\in \X(H)^+$ and $P\in \gap(U)$.
Suppose that $P\notin U$ and $P\notin \gap(\emptyset)$.  Then
there exists $Q\in \X(H)$ such that $\tau^{-1}[P]\subseteq Q$,
which implies that $Q\in U$, because $P\in \gap(U)$. The fact that
$P\in \gap(U)$ together with $P\notin U$ implies that
$\tau^{-1}[P]\nsubseteq P$. Thus, there exists $a\in H$ such that
$\gamma(a)\in P$ and $a\notin P$. It follows from Lemma \ref{gr}
that $\varphi(\gamma(a)) = \varphi(a) \cup (\X(H))_M$. Since
$\gamma(a)\in P$ and $a\notin P$ then $P\in (\X(H))_M$. But
$P\subseteq Q$, so $P = Q$. However, $P\notin U$ and $Q\in U$,
which is a contradiction. Hence, $\gap(U) \subseteq U \cup
\gap(\emptyset)$. Then $\gap$ is the function $\gamma$ on
$\X(H)^+$, and therefore $(\X(H)^+, \gap) \in \gaIS$.

Finally we  see that $\gap(U) = U\cup (\X(H))_M$ for every
$U\in \X(H)^+$. Let $U\in \X(H)^+$. Since $\gap(U) \subseteq U
\cup \gap(\emptyset)$ and $\gap(\emptyset) = (\X(H))_M$ then
$\gap(U) \subseteq U \cup (\X(H))_M$. On the other hand,
$U\subseteq \gap(U)$ and $(\X(H))_M = \gap(\emptyset) \subseteq
\gap(U)$ because $\emptyset \subseteq U$. Hence, $U \cup (\X(H))_M
\subseteq \gap(U)$. Therefore, we obtain that $\gap(U) = U\cup
(\X(H))_M$.
\end{proof}

In particular, we have that if $(H,\gamma) \in \gaHil$, then
$(\FC(H),\gap)\in \gaIS$. This follows from the facts that, being
$\gap$ a frontal operator, $\FC(H)$ is closed under $\gap$ and
that in a Heyting algebra the existence of $\gamma$ is equivalent
to the existence of a frontal operator $\tau$ which satisfies that
$\neg \tau(0) = 0$ and $\tau(a) \leq a \vee \tau(0)$ for every
$a$.

The following proposition follows from Corollary \ref{cor0} and
Lemma \ref{grb}.

\begin{prop}\label{propgr}
The functor $\fro:\FHil_0 \ra \FIS_0$ can be restricted to a
functor from $\fro: \gaHil \ra \gaIS$.
\end{prop}

We also write $\U$ for the forgetful functor from $\gaIS$ to
$\gaHil$. The following result follows from Corollary \ref{cor1}
and Proposition \ref{propgr}.

\begin{cor}
The functor $\fro: \gaHil \ra \gaIS$ is left adjoint to $\U$.
\end{cor}

\section{An adjunction between $\SHil$ and $\SIS$}\label{SHil}

In \cite{Kuz} Kuznetsov introduced a new operation on Heyting
algebras as an attempt to build an intuitionistic version of the
provability logic of G\"{o}del-L\"{o}b, which formalizes the
concept of provability in Peano Arithmetic. This unary operation,
which we shall call successor, was also studied by Caicedo and
Cignoli in \cite{Caici} and by Esakia in \cite{Esakia}. A unary
operation $S$ on a Heyting algebra is a \emph{successor} operation
if it satisfies for very $a \in H$ the following conditions:
\begin{enumerate}
\item $a\leq S(a)$, \item $S(a)\leq b\vee (b\ra a)$, \item
$S(a)\rightarrow a = a$.
\end{enumerate}
 The conditions $a\leq S(a)$ and
$S(a)\ra a \leq S(a)$ can be replaced by $S(a)\ra a = a$, as it
was showed in \cite{CSSM}. In \cite[Proposition 2.3]{CSSM} it was
proved  that in fact a unary map $S$ on a Heyting algebra $H$ is a
successor operation if and only if it is frontal operator that
satisfies $S(a)\ra a = a$ for every $a \in H$.

Moreover, in \cite{CSSM} it was also proved that a unary
operation $S$ on a Heyting algebra $H$ is a successor operation if
and only if for every $a \in H$ it holds that
\[
S(a) =\;\text{min}\{b:b\ra a\leq a\}.
\]
Therefore, if a successor operation exist on a Heyting algebra,
then it is unique. The successor exists in all finite Heyting
algebras (see \cite{Caici}), but there are examples of Heyting
algebras where there is no successor operation.
\vspace{3pt}

Let $(P,\leq)$ be a poset. We know that $P^+$ is a complete
Heyting algebra. In what follows we will see that the existence of
a successor function in $P^+$ can be easily described in terms of
a certain condition on $(P,\leq)$. This result provides examples
of complete Heyting algebras without a successor function.

Recall that a  poset  $(P,\leq)$  satisfies the ascending chain
condition (ACC) if every strictly ascending sequence of elements
eventually terminates. Equivalently, given any sequence $x_{1}
\leq x_{2} \leq x_{3}\leq \ldots$ there exists a natural number
$n$ such that $x_n = x_m$ for every $m\geq n$. It is known that
the (ACC) is equivalent to the following condition: every nonempty
subset of $P$ has a maximal element. Notice that straightforward
computations show that the (ACC) is also equivalent to the
following condition, which will be called (P): for every downset
$V$ of $(P,\leq)$, if $x\in V$,  then there exists $y\in V_M$ such
that $x\leq y$. Moreover, the condition (P) is equivalent to the
following one: for every downset $V$ of $(P,\leq)$ it holds that
$V = (V_M]$.

Let $(P,\leq)$ be a poset. Consider the co-derivative frontal
operator  $\tau:P^+ \ra P^+$, that, as we saw in Corollary
\ref{pmots}, satisfies that $\tau(U) = U \cup (U^c)_M$ for every $U
\in P^{+}$. It follows  that  for every $U \in P^{+}$,  $\tau(U)
\Ra U = ((U^c)_M]^c$. Therefore, $\tau$ is a successor function on
$P^+$ if and only if $(P,\leq)$ satisfies the (ACC). This property
can be also obtained from results in \cite{Kuz79}.

It naturally arises the following question: is the (ACC) satisfied
when $P^+$ has successor function? Next theorem answers it in the
positive.

\begin{thm}\label{succ-ACC}
Let $(P, \leq)$ be a poset. Then $P^{+}$ has a successor if and
only if $(P, \leq)$ satisfies $\mathrm{(ACC)}$.
\end{thm}

\begin{proof}
It follows from the discussion above that if  $(P,\leq)$ satisfies
the (ACC), then $P^+$ has successor $S$ and this takes the form
$S(U) = U\cup (U^c)_M$ for every $U \in P^{+}$.

Conversely, suppose that $P^+$ has successor $S$. Let us consider
the co-derivative frontal operator $\tau:P^+ \ra P^+$. To see that
$(P,\leq)$ has the (ACC) it is enough to see that $\tau = S$. To
this end we prove first the following claims.

\textsc{Claim 1}. If $f$ is a frontal operator on $P^+$, then
$f(U) \subseteq S(U)$ for every $U\in P^+$. In order to prove it,
let $U\in P^+$. Then, since $f$ is a frontal operator we have
$f(U) \subseteq S(U) \cup (S(U) \Ra U) = S(U) \cup U = S(U)$.

\textsc{Claim 2}. If $\{U_{i}\}_{i\in I}\subseteq P^+$ then
$S(\bigcap_{i\in I} U_i) \subseteq \bigcap_{i\in I} S(U_{i})$.
This holds  because $S$ is a monotone map.

\textsc{Claim 3}. In $P^+$ the operator $\tau$ preserves arbitrary
meets. In order to show it, let $\{U_{i}\}_{i\in I} \subseteq
P^+$. The monotonicity of $\tau$ implies that $\tau (\bigcap
_{i\in I} U_{i}) \subseteq \bigcap_{i\in I} \tau(U_{i})$.
Conversely, let $x\in \bigcap_{i\in I} \tau(U_{i})$. If $x\in
\bigcap _{i\in I} U_{i}$, then $x \in \tau(\bigcap _{i\in I}
U_{i})$. On the contrary, if $x\notin \bigcap _{i\in I} U_{i}$,
let us see that $x \in  ((\bigcap _{i\in I} U_{i})^{c})_M$.  Tho
this end suppose that $x \leq y$ with $y \not \in \bigcap _{i\in
I} U_{i}$. Thus  there exists $j \in I$ such that  $y\notin
U_{j}$, and being $U_j$ an upset we have $x \not \in U_j$. Hence,
since $x \in \tau(U_j)$ then $x \in (U_{j}^{c})_M$. Therefore, $x
= y$. We conclude that $\tau (\bigcap _{i\in I} U_{i}) =
\bigcap_{i\in I} \tau(U_{i})$.

\textsc{Claim 4}: For every $x\in P$, $S((x]^{c}) =
\tau((x]^{c}) = (x]^{c} \cup \{x\}$. In order to prove it, first
note that $\tau((x]^{c}) \Ra (x]^c = (x]^c$. On the other hand, by
Claim 1 we know that $\tau((x]^{c}) \subseteq S((x]^{c})$. Using
that $\tau$ is a frontal operator on the Heyting algebra $P^+$ and
the fact that $\tau((x]^{c}) \Ra (x]^c = (x]^c$ we will prove that
$S((x]^{c}) \subseteq \tau((x]^{c})$ as follows:
\[
\begin{array}
[c]{lllll}
S((x]^{c})& \subseteq    & \tau((x]^{c}) \cup (\tau((x]^{c}) \Ra (x]^{c})&  & \\
        & =              & \tau((x]^{c}) \cup (x]^c   &  & \\
        & =              & \tau((x]^{c}).& &
\end{array}
\]
Thus, $S((x]^{c}) = \tau((x]^{c})$.

Now we will use the previous claims to show that for every $U\in
P^+$ we have that $\tau(U) = S(U)$. To this end we first note that
since obviously for every downset $V$ it holds that $V=
\bigcup_{x\in V} (x]$. Thus, for very upset $U$ we have that $U =
\bigcap_{x \in U^{c}}(x]^{c}$. Let $U\in P^+$. Then
\[
\begin{array}
[c]{lllll}
\tau(U)& \subseteq    & S(U)&  & \\
        & =           & S(\bigcap_{x\in U^c}(x]^c)   &  & \\
        & \subseteq   & \bigcap_{x\in U^c}S((x]^c)   &  & \\
        & =           & \bigcap_{x\in U^c}\tau((x]^c)   &  & \\
        & =           & \tau(\bigcap_{x\in U^c}(x]^c)   &  & \\
        & =           & \tau(U).& &
\end{array}
\]
Therefore, $\tau (U) = S(U)$, which was our aim.
\end{proof}

Theorem \ref{succ-ACC} implies that not every Heyting algebra has
a successor. Indeed, if $(P,\leq)$ is a poset that does not
satisfy the (ACC), then the Heyting algebra $P^{+}$ has no
successor. This can be used to give an example of a Heyting
algebra with successor $(H,S)$ such that the Heyting algebra
$\X(H)^+$  of the upsets of the partial order of the irreducible
implicative filters of  $H$ (which are the prime filters of $H$)
does not have successor.  For instance, if $\mathbb{N}^{0}$ is the
set of natural numbers with its inverse order, and $\oplus$ is the
ordinal sum of posets (see \cite{BD}, p. 39), then
$\mathbb{N}\oplus\mathbb{N}^{0}$ is a Heyting algebra with
successor. However, $(\X(\mathbb{N}\oplus\mathbb{N}^{0}))^{+}$ is
not a Heyting algebra with successor because the (ACC) is not
satisfied in the poset $\X(\mathbb{N}\oplus\mathbb{N}^{0})$. The
fact that $(\X(\mathbb{N}\oplus\mathbb{N}^{0}))^{+}$ is not a
Heyting algebra with successor was also mentioned in \cite{CSMam}.
\vspace{3pt}

In \cite{CSM0} the successor operation defined on Heyting algebras
was generalized to the setting of Hilbert algebras.

Let $H\in \Hil$. A unary function $S:H\ra H$ is a \emph{successor}
operation if for every $a, b \in H$ the following conditions are
satisfied:
\begin{description}
\item[(S1)] $S(a) \leq ((b\ra a)\ra b)\ra b$, \item[(S2)] $S(a)\ra
a \leq S(a)$.
\end{description}
for every $a, b \in H$.

It follows from \cite[Section 3, Corollary 4]{CSM0} and
\cite[Section 3, Proposition 5]{CSM0} that a unary function $S$ on
a Hilbert algebra $H$ is a successor operation if and only if it
is a frontal operator that satisfies the equality
\[
S(a)\ra a = a
\]
for every $a\in H$. Moreover, it was also proved in \cite{CSM0}
that a unary map $S$ on a Hilbert algebra is a successor operation
if and only if for every $a \in H$,
\[
S(a) =\;\text{min}\; \{b\in H:b\ra a \leq b\}.
\]
Thus, if there exist a successor operation on a  Hilbert algebra,
then it is unique. There are examples of finite Hilbert algebras
where no successor operation exists, see for instance
\cite[Section 3, Example 14]{CSM0}.

\begin{defn}
We say that an algebra $(H,S)$ is a \emph{Hilbert algebra with
successor} if $H$ is a Hilbert algebra and $S$ is a successor
operation.
\end{defn}

In a similar way we define implicative semilattices with
successor. We write $\SHil$ for the algebraic category of Hilbert
algebras with successor and $\SIS$ for the algebraic category of
implicative semilattices with successor.

Let $H\in \IS$. For every $a\in H$ we define the set
\[
S_{a} = \{b\in H:b\ra a\leq b\}.
\]

\begin{prop}\label{Sfil}
Let $H\in \IS$. For every $a\in H$ the set $S_{a}$ is a filter.
Moreover, if $H$ is finite then there exists the minimum of $S_a$
for every $a\in H$, i.e.,  there exists a successor function.
\end{prop}

\begin{proof}
Suppose that $a\in H$. It is immediate that $1\in S_a$. We will
prove that $S_a$ is an upset. Let $b\leq c$ and $b\in S_a$. Thus,
$b\ra a \leq b$ and $c \ra a \leq b\ra a$. Since $b\leq c$ then
$c\ra a \leq c$, i.e., $c\in S_a$. Thus, $S_a$ is an upset.

We proceed to show that if $b,c\in S_a$, then $b\we c\in S_a$. Let
$b,c\in S_a$, i.e., $b\ra a \leq b$ and $c\ra a\leq c$. Note that
$b\we ((b\we c) \ra a) \leq c\ra a \leq c$, so $b \we ((b\we c)\ra
a)\leq b\we c$. Besides we have that $b\we ((b\we c)\ra a) \leq
(b\we c)\ra a$. Thus, $b\we ((b\we c)\ra a) \leq a$ and hence
$(b\we c)\ra a \leq b\ra a$. Since $b\ra a\leq b$, we obtain that
$(b\we c)\ra a \leq b$. Using the analogous argument we also
obtain that $(b\we c)\ra a \leq c$. Hence, $(b\we c)\ra a \leq b
\we c$ and therefore $b\we c \in S_a$.
\end{proof}

The following lemma is \cite[Section 4, Lemma 8]{CSM0}.

\begin{lem} \label{lmot}
If $(H,S)\in \SHil$, then $\varphi(S(a)) = \varphi(a) \cup
(\varphi(a)^{c})_M$ for every $a\in H$.
\end{lem}

The next lemma will be used later.

\begin{lem} \label{ps1}
Let $(P,\leq)$ be a poset and $U_1,\ldots,U_n\in P^+$. Then
\[
\bigcap_{i=1}^{n}(U_i \cup (U_{i}^{c})_M) = (\bigcap_{i=1}^{n}
U_i) \cup (\bigcup_{i=1}^{n} U_{i}^{c})_M.
\]
\end{lem}

\begin{proof}
Consider a poset $(P,\leq)$. It follows from Corollary \ref{pmots}
that the co-derivative frontal operator $\tau$ of the Heyting
algebra $P^{+}$ is such that $\tau(U) = U \cup (U^{c})_M$, for
every $U \in P^{+}$. Since $\tau$ is a frontal operator on the
Heyting algebra $P^+$, then $\tau$ preserves finite meets.
Therefore, we obtain the desired result.
\end{proof}

The following lemma is a generalization of \cite[Lemma 7]{D}. To
prove it we introduce the following definition. Let $H\in \Hil$
and $a\in H$. We define the set
\[
\sigma(a) := \{F\in \Fil(H):a\in F\}.
\]

\begin{lem} \label{ps3}
Let $H$ be a Hilbert algebra and $a_1,\ldots,a_n\in H$. Then
\[
\bigcup_{i=1}^{n} \varphi(a_i)^c = ((\bigcup_{i=1}^{n}
\varphi(a_i)^c)_M].
\]
\end{lem}

\begin{proof}
Assume that  $P \in \bigcup_{i=1}^{n} \varphi(a_i)^c$. Consider the set
\[
\Sigma = \{Q\in \Fil(H):P\subseteq Q\;\text{and}\;Q\in
\bigcup_{i=1}^{n} \sigma(a_i)^c\}.
\]
Since $P\in \Sigma$ then $\Sigma \neq \emptyset$. Straightforward
computations show that if $\{Q_i\}_{i\in I}$ is a chain of
elements in $\Sigma$, then $\bigcup_{i \in I}Q_i \in \Sigma$, so
by Zorn's lemma there exists a maximal element in $\Sigma$, which
will be denoted by $Q$. In what follows we see that $Q\in \X(H)$
by using Lemma \ref{SC}. Let $a,b\notin Q$. We prove that there
exists $c\notin Q$ such that $a\leq c$ and $b\leq c$. Assume the
contrary, i.e., that for every $c\in H$ it holds that if $a\leq c$
and $b\leq c$, then $c\in Q$. Thus, $[a)\cap [b) \subseteq Q$,
i.e., $Q = Q \vee ([a)\cap [b))$. Since the lattice of implicative
filters of $H$ is distributive (\cite[Theorem 6]{D}) then $Q =
(Q\vee [a))\cap (Q\vee [b))$ (note that $Q\vee [a) = F(Q \cup
\{a\})$ and $Q\vee [b) = F(Q \cup \{b\})$). Since $P\subseteq
Q\subset Q\vee [a)$ and $P\subseteq Q\subset Q\vee [b)$ then it
follows from the maximality of $Q$ that $Q\vee [a), Q\vee [b)
\notin \Sigma$. In particular, we obtain that $a_1,\ldots,a_n \in
(Q\vee [a))\cap (Q\vee [b)) = Q$, which is a contradiction. Hence,
$Q\in \X(H)$ and, moreover, $Q\in \bigcup_{i=1}^{n}
\varphi(a_i)^c$. We prove that $Q$ is maximal in this set. To this
end, assume that $Q'\in \X(H)$ is such that $Q\subseteq Q'$ and
$Q'\in \bigcup_{i=1}^{n} \varphi(a_i)^c$. In particular, $Q'\in
\bigcup_{i=1}^{n} \sigma(a_i)^c$, so $Q'\in \Sigma$. Thus, $Q =
Q'$. Therefore, $P\subseteq Q$ with $Q\in (\bigcup_{i=1}^{n}
\varphi(a_i)^c)_M$, and hence $P \in((\bigcup_{i=1}^{n}
\varphi(a_i)^c)_M]$. We conclude that $\bigcup_{i=1}^{n}
\varphi(a_i)^c \subseteq ((\bigcup_{i=1}^{n} \varphi(a_i)^c)_M]$.
The converse inclusion is immediate.
\end{proof}

It is immediate to see that for every poset $(P,\leq)$ and $U\in
P^+$,  $U \cup (U^{c})_M \in P^+$.

\begin{lem} \label{ps4}
Let $(H,S) \in \SHil$. Then $(\X(H)^+, \Sp) \in \SIS$. Moreover,
$\Sp$ takes the form $\Sp(U) = U \cup (U^c)_M$.
\end{lem}

\begin{proof}
Suppose that $U\in \FC(H)$; so there exist $a_1,\ldots,a_n\in H$ such that
$U = \bigcap_{i=1}^{n} \varphi(a_i)$. The frontal operator $\Sp$
on $\X(H)^{+}$ applied to $U$ gives that
\[
\Sp(U) = \bigcap_{i = 1}^{n} \varphi(S(a_i)).
\]
Then Lemma \ref{lmot} implies
\[
\Sp(U) = \bigcap_{i = 1}^{n} (\varphi(a_i) \cup (\varphi(a_i)^{c})_M).
\]
Now using Lemma \ref{ps1} we obtain that
\[
\Sp(U) = \bigcap_{i = 1}^{n} \varphi(a_i) \cup \bigcap_{i = 1}^{n} (\varphi(a_i)^{c})_M.
\]
Therefore, $\Sp(U) = U \cup (U^{c})_M$. Since $\Sp$ is a frontal
operator, to show that it is a successor we only need to prove
that $\Sp(U)\Ra U = U$. This holds if and only if $(U^{c}_M]^c =
U$. But it follows from Lemma \ref{ps3} that $(U^{c}_M]^c = U$.
Therefore, the implicative semilattice $\FC(H)$ has successor
$\Sp$ and  $\Sp(U) = U \cup (U^{c})_M$.
\end{proof}

The following proposition follows from Proposition \ref{ps5} and
Lemma \ref{ps4}.

\begin{prop}\label{propSuc}
The functor $\f:\Hil \ra \IS$ can be extended to a functor from
$\fro: \SHil \ra \SIS$.
\end{prop}

We also write $\U$ for the forgetful functor from $\SIS$ to
$\SHil$. The following corollary follows from Theorem
\ref{HiltoIS} and Proposition \ref{propSuc}.

\begin{cor}
The functor $\fro: \SHil \ra \SIS$ is left adjoint to $\U$.
\end{cor}

\section{An adjunction between $\GabHil$ and
$\GabIS$}\label{GabHil}

In \cite[Example 5.3]{Caici} Caicedo and Cignoli studied an
example of implicit compatible operation of Heyting algebras,
which was considered by Gabbay in \cite{Gabbay1}. When it exists,
it is also a case of a frontal operator.

Let $H$ be a Heyting algebra. A unary map $G$ on $H$ is a
\emph{Gabbay function} ($G$-function for short) if the following
conditions are satisfied for every $a, b \in H$:
\begin{enumerate}
\item $G(a)\leq b \vee (b\ra a)$, \item $a\ra b\leq G(a)\ra G(b)$,
\item $a\leq G(a)$, \item $G(a)\leq \neg \neg a$, \item  $G(a)\ra
a \leq \neg \neg a \ra a$.
\end{enumerate}
In \cite{CSSM} it was proved that a unary map $G$ on a Heyting
algebra is a $G$-function if and only if for every $a \in H$,
\[
G(a) = \;\text{min}\;\{b\in H:(b\ra a) \we \neg \neg a \leq b\}.
\]
Thus, if there is a $G$-function on a Heyting algebra,  then it is
unique. In every finite Heyting algebra there exists the function
$G$ (see \cite{Caici}). However, there are examples of Heyting
algebras where no Gabbay function exists. In \cite[Proposition
2.7]{CSSM} it was also proved that a $G$-function of a Heyting
algebra is a frontal operator.

In \cite{CSM0} the notion of Gabbay function was generalized to
the framework of bounded Hilbert algebras. Let $H\in \Hil_0$. We
say that function $G:H\ra H$ is a $G$-\emph{function} if the
inequalities $\mathbf{(i2)}$, $\mathbf{(i3)}$ hold as well as the
following additional ones:
\begin{description}
\item[(G4)] $Ga\leq \neg \neg a$, \item[(G5)] $Ga \ra a \leq \neg
\neg a\ra a$.
\end{description}
Let $H\in \Hil_0$. It follows from \cite[Section 3, Corollary
19]{CSM0} that a unary map $G$ is a $G$-function if and only if it
is a  frontal operator that satisfies the additional conditions
$\textbf{(G4)}$ and $\textbf{(G5)}$. It is interesting to note
that in \cite[Section 3, Proposition 10]{CSM0} it was proved that
a $G$-function is also characterized by the fact that for every $a
\in H$
\[
G(a) = \;\text{min}\; \{b\in H: b\ra a\leq \neg \neg a \ra b\}.
\]
There are examples of finite Hilbert algebras where no
$G$-function exists, see for instance \cite[Section 3, Example
15]{CSM0}.

We write $\GabHil$ for the algebraic category whose objects are
algebras $(H,G)$, where $H\in \Hil_0$ and $G$ is a $G$-function.
In a similar way we define $\GabIS$.\vspace{10pt}

Let $H\in \Hil_0$. For every $a\in H$ we define $G_{a} := \{b\in
H:b\ra a\leq \neg \neg a \ra b\}$. If $H\in \IS_0$ then $G_a =
\{b\in H:\neg \neg a \we (b\ra a)\leq b\}$.

\begin{prop}
Let $H\in \IS_0$. For every $a\in H$ the set $G_{a}$ is a filter.
Moreover, if $H$ is finite, then there exists the minimum of $G_a$
for every $a\in H$, i.e., there exists a $G$-function.
\end{prop}

\begin{proof}
Let $a\in H$. It is immediate that $1\in G_a$. We will prove that
$G_a$ is an upset. Let $b\leq c$ and $b\in G_a$. Then $\neg \neg a
\we (b\ra a) \leq b$ and $c \ra a \leq b\ra a$. Taking into
account that $b\leq c$ we deduce that $\neg \neg a \we (c\ra a)
\leq c$, i.e., $c\in G_a$. Thus, $G_a$ is an upset.

We proceed to show that if $b,c\in G_a$ then $b\we c\in
G_a$. Let $b,c\in G_a$, i.e., $\neg \neg a \we (b\ra a) \leq b$
and $\neg \neg a \we (c\ra a)\leq c$. Note that
\[
\neg \neg a \we (b\we ((b\we c) \ra a)) \leq \neg \neg a \we (c\ra
a) \leq c,
\]
so we get
\[
\neg \neg a \we b \we ((b\we c)\ra a)\leq b\we c.
\]
Moreover, $\neg \neg a \we b\we ((b\we c)\ra a) \leq (b\we c)\ra
a$. Thus, $\neg \neg a \we b\we ((b\we c)\ra a) \leq a$. It follows that
$\neg \neg a \we ((b\we c)\ra a) \leq b\ra a$. Hence, using that $b \in G_a$,
\[
\neg \neg a \we ((b\we c)\ra a) \leq \neg \neg a \we (b\ra a) \leq
b,
\]
Using the analogous
argument, we have $\neg \neg a \we ((b\we c)\ra a) \leq c$. Therefore,
$\neg \neg a \we ((b\we c)\ra a) \leq b \we c$, i.e., $b\we c \in
G_a$.
\end{proof}

The following is \cite[Section 3, Lemma 17]{CSM0}.

\begin{lem} \label{Gabr}
Let $(H,G) \in \GabHil$. Then $\varphi(G(a)) = \varphi(a) \cup
(\varphi(\neg \neg a)\cap(\varphi(a)^{c})_M)$ for every $a\in H$.
\end{lem}

Let $H$ be a bounded Hilbert algebra and $a\in H$. We have that
$\varphi(\neg a) = \varphi(a) \Ra \varphi(0) = \varphi(a) \Ra
\emptyset$. In what follows we write $\neg \varphi(a)$ in place of
$\varphi(a) \Ra \emptyset$.
\vspace{1pt}

Let $(P,\leq)$ be a poset and let $U\in P^+$. Note that $U \cup
((\neg \neg U) \cap (U^c)_M) \in P^+$ because $U \cup ((\neg \neg
U) \cap (U^c)_M) = (U \cup (U^c)_M)\cap \neg \neg U$ and $U \cup
(U^c)_M \in P^+$.

\begin{lem} \label{g1}
If $(H,G) \in \GabHil$, then $(\X(H)^+, \Gp) \in \GabIS$.
Moreover, $\Gp$ takes the form $\Gp(U) = U \cup ((\neg \neg U)
\cap (U^c)_M)$.
\end{lem}

\begin{proof}
Assume that  $U\in \FC(H)$, so there exist $a_1,\ldots,a_n\in H$ such that
$U = \varphi(a_1)\cap \cdots \cap \varphi(a_n)$. In what follows we
will see that $\Gp(U) =U \cup ((\neg \neg U) \cap (U^c)_M)$.

It follows from Lemma \ref{Gabr} that $\Gp(G(a_i)) =
\varphi(a_i)\cup [\varphi(\neg \neg a_i) \cap
(\varphi(a_i)^{c})_M)]$ for every $i=1,\ldots,n$. Hence, using
 Lemma \ref{ps1} we have
\[
\begin{array}
[c]{lllll}
\Gp(U)& = & \bigcap_{i=1}^{n} \varphi(G(a_i))&  & \\
        & =    & \bigcap_{i=1}^{n}(\varphi(a_i)\cup (\varphi(a_i)^{c})_M) \cap \varphi(\neg \neg a_i)   &  & \\
        & =    & \bigcap_{i=1}^{n} (\varphi(a_i)\cup (\varphi(a_i)^{c})_M)\cap \varphi(\neg \neg a_i)   &  & \\
        & =    & (\bigcap_{i=1}^{n} \varphi(a_i)\cup (\varphi(a_i)^{c})_M) \cap (\bigcap_{i=1}^{n} \neg \neg \varphi(a_i))   &  & \\
        & =    & (U \cup (U^c)_M) \cap \neg \neg (\bigcap_{i=1}^{n} \varphi(a_i))   &  & \\
        & =    & (U \cup (U^c)_M) \cap \neg \neg U.& &
\end{array}
\]
Therefore, $\Gp(U) = U \cup ((U^c)_M \cap \neg \neg U)$.

Now we prove that $\Gp$ is a $G$-function. Let $U\in \FC(H)$. Then
$\Gp(U) = U \cup ((\neg \neg U) \cap (U^c)_M) \subseteq U \cup
\neg \neg U = \neg \neg U$, so $\Gp(U) \subseteq \neg \neg U$.
Finally we will need to prove that $\Gp(U) \Ra U \subseteq \neg
\neg U \Ra U$, i.e., $(\neg \neg U \cap U^c] \subseteq (\neg \neg
U \cap (U^c)_M]$. This inclusion follows from Lemma \ref{ps3}.
\end{proof}

The following proposition follows from Corollary \ref{cor0} and
Lemma \ref{g1}.

\begin{prop}\label{propGab}
The functor $\fro:\FHil_0 \ra \FIS_0$ can be restricted to a
functor from $\fro: \GabHil \ra \GabIS$.
\end{prop}

We also write $\U$ for the forgetful functor from $\GabIS$ to
$\GabHil$. The following result follows from Corollary \ref{cor1}
and Proposition \ref{propGab}.

\begin{cor}
The functor $\fro: \GabHil \ra \GabIS$ is left adjoint to $\U$.
\end{cor}

As in the case of Heyting algebras with successor, we have that
$\mathbb{N}\oplus\mathbb{N}^{0}$ is a Heyting algebra with a
$G$-function. It can be also proved that
$(\X(\mathbb{N}\oplus\mathbb{N}^{0}))^{+}$ is not a Heyting
algebra with a $G$-function.

\subsection*{Acknowledgments}
This project has received funding from the European Union's
Horizon 2020 research and innovation program  under the Marie
Sklodowska-Curie grant agreement No. 689176.

The first author was also partially supported by the research
grant 2014 SGR 788 from the government of Catalonia and by the
research projects MTM2016-74892-P  from the
government of Spain, which includes \textsc{feder} funds from the
European Union and he also acknowledges financial support from the
Spanish Ministry of Economy and Competitiveness, through the
``Mar\'ia de Maeztu'' Programme for Units of Excellence in R\&D
(MDM-2014-0445). The second author was also supported by CONICET
Project PIP 112-201501-00412.

Both authors declare that they have no conflict of interest. This
article does not contain any studies with animals or humans
performed by any of the authors.

\quad

\noindent ---------------------------------------------------------------------------------------
\\
Ramon Jansana,\\
Departament de Filosofia,\\
Universitat de  Barcelona.\\
Montalegre, 6,\\
08001, Barcelona,\\
España.\\
jansana@ub.edu

\noindent  -----------------------------------------------------------------------------------------
\\
Hern\'an Javier San Mart\'in,\\
Departamento de Matem\'atica, \\
Facultad de Ciencias Exactas (UNLP), \\
and CONICET.\\
Casilla de correos 172,\\
La Plata (1900),\\
Argentina.\\
hsanmartin@mate.unlp.edu.ar

\end{document}